\documentclass[11pt]{amsart}
\usepackage{amsfonts,amssymb,amsthm}
\usepackage{amsmath,amscd}
\usepackage{pstricks}
\usepackage{pstricks,pst-node}
\usepackage{mathrsfs}
\usepackage[all]{xy}

\DeclareMathAlphabet{\mathpzc}{OT1}{pzc}{m}{it}

\theoremstyle{plain}
\newtheorem{Thm}{Theorem}[section]
\newtheorem{Prop}[Thm]{Proposition}

\newtheorem{Lem}[Thm]{Lemma}
\newtheorem{Coro}[Thm]{Corollary}
\theoremstyle{definition}
\newtheorem{Def}[Thm]{Definition}

\numberwithin{equation}{section}

\newcommand{\Rep}{\mathrm{Rep}}
\newcommand{\ch}{\mathrm{ch}}
\newcommand{\kav}{\kappa_{v}}
\newcommand{\wt}{\mathrm{wt}}
\newcommand{\wtl}{\mathrm{wt}_{\ell}}
\newcommand{\Cr}{\widehat{\mathcal C}_r}
\newcommand{\hFn}{\widehat{\sF}_n}
\newcommand{\hFnchi}{\widehat\sF_{n,\chi}}
\newcommand{\hFnchij}{\widehat\sF_{n,\chi_j}}

\newcommand{\hFnchio}{\widehat\sF_{n,\chi_1}}
\newcommand{\hFnchit}{\widehat\sF_{n,\chi_2}}
\newcommand{\hFnchiot}{\widehat\sF_{n,\chi_1\chi_2}}

\newcommand{\hFnchii}{\widehat\sF_{n,\chi_i}}
\newcommand{\hFnr}{\widehat\sF_{n,r}}

\newcommand{\hFnQz}{\widehat\sF_{n,\ol{\bfQ_0}}}
\newcommand{\hFnQ}{\widehat\sF_{n,\ol{\bfQ}}}
\newcommand{\hFnrQ}{\widehat\sF_{n,r,\ol{\bfQ}}}
\newcommand{\hFnQot}{\widehat\sF_{n,\ol{\bfQ_1}\cdot\ol{\bfQ_2}}}

\newcommand{\hFnrchi}{\widehat\sF_{n,r,\chi}}

\newcommand{\bse}{\boldsymbol{e}}
\newcommand{\bsa}{\boldsymbol{a}}

\newcommand{\bfa}{{\mathbf{a}}}
\newcommand{\bfb}{{\mathbf{b}}}

\newcommand{\bfi}{{\mathbf{i}}}
\newcommand{\bfj}{{\mathbf{j}}}

\newcommand{\bfs}{{\mathbf{s}}}

\newcommand{\bfP}{{\mathbf{P}}}
\newcommand{\bfQ}{{\mathbf{Q}}}

\def\fS{{\frak S}}

\newcommand{\ms}{\mathscr}

\newcommand{\sfz}{{\mathsf z}}
\newcommand{\sfs}{{\mathsf s}}

\newcommand{\sfF}{{\mathsf F}}

\def\sC{{\mathcal C}}

\def\sF{{\mathcal F}}

\def\sH{{\mathcal H}}

\def\sM{{\mathcal M}}

\def\sX{{\mathcal X}}
\def\sY{{\mathcal Y}}

\newcommand{\mbzn}{\mathbb Z^{n}}
\newcommand{\mbnn}{\mathbb N^{n}}
\newcommand{\mbn}{\mathbb N}

\newcommand{\mbc}{\mathbb C}
\newcommand{\mbz}{\mathbb Z}

\newcommand{\mba}{\mathbb A}

\newcommand{\ttk}{\mathtt{k}}
\newcommand{\tth}{\mathtt{h}}

\newcommand{\ttx}{\mathtt{x}}
\newcommand{\ttg}{\mathtt{g}}

\newcommand{\End}{\operatorname{End}}
\newcommand{\Hom}{\operatorname{Hom}}
\newcommand{\Ext}{\operatorname{Ext}}

\newcommand{\spann}{\operatorname{span}}

\newcommand{\Det}{\mathrm{Det}}

\newcommand{\la}{{\lambda}}
\newcommand{\La}{\Lambda}
\newcommand{\ga}{{\gamma}}
\newcommand{\Ga}{{\Gamma}}

\newcommand{\dt}{\delta}
\newcommand{\Dt}{\Delta}

\newcommand{\Og}{\Omega}
\newcommand{\og}{\omega}

\newcommand{\vi}{\varphi}

\newcommand{\al}{\alpha}
\newcommand{\bt}{\beta}
\newcommand{\sg}{\sigma}

\newcommand{\ol}{\overline}

\newcommand{\Lcp}{\bar L}

\newcommand{\Mcp}{\bar M}
\newcommand{\Icp}{\bar I}
\newcommand{\Wcp}{\bar W}
\newcommand{\Jcp}{\bar J}

\def\ggp#1#2{\left[\kern-3.2pt\left[{#1\atop #2}\right]\kern-3.2pt\right]}

\def\leq{\leqslant}\def\geq{\geqslant}
\def\le{\leqslant}

\newcommand{\bop}{\bigoplus}

\newcommand{\ot}{\otimes}

\newcommand{\han}{\subseteq}

\newcommand{\h}{\widehat}
\newcommand{\ti}{\widetilde}

\newcommand{\Lanr}{\Lambda(n,r)}

\newcommand{\lra}{\longrightarrow}
\newcommand{\ra}{\rightarrow}

\newcommand{\zr}{\zeta_r}

\newcommand{\vtg}{{\!\vartriangle\!}}

\def\ttv{v}

\newcommand{\Hrv}{\sH(r)_v}
\newcommand{\afHrv}{\sH_{\vtg}(r)_v}

\newcommand{\afUglv}{U_{v}(\widehat{\frak{gl}}_n)}
\newcommand{\Uglv}{U_{v}({\frak{gl}}_n)}

\newcommand{\fSr}{\fS_r}

\newcommand{\afSrv}{{\mathcal S}_{\vtg}(n,r)_{v}}
\newcommand{\afSrprimev}{{\mathcal S}_{\vtg}(n,r')_{v}}

\newcommand{\afUslv}{U_{v}(\widehat{\frak{sl}}_n)}

\newcommand{\Xn}{\mathsf X(n)}
\newcommand{\Qnp}{\mathsf Q^+(n)}
\newcommand{\Qnpr}{\mathsf Q^+(n,r)}
\newcommand{\Qnpri}{\mathsf Q^+(n,r_i)}
\newcommand{\Qnpt}{\mathsf Q^+(n,t)}
\newcommand{\Rn}{\mathsf R(n)}
\newcommand{\Rnp}{\mathsf R^+(n)}
\newcommand{\Rnm}{\mathsf R^-(n)}
\newcommand{\Pn}{\mathsf P(n)}
\newcommand{\Pno}{\mathsf P(n+1)}
\newcommand{\Pnop}{\mathsf P^+(n+1)}
\newcommand{\Pnp}{\mathsf P^+(n)}

\newcommand{\Qn}{\mathsf Q(n)}

\newcommand{\Ogv}{\Og_{v}}

\makeatletter
\makeatletter
\def\section{\def\@secnumfont{\mdseries}\@startsection{section}{1}%
  \z@{.7\linespacing\@plus\linespacing}{.5\linespacing}%
  {\normalfont\scshape\centering}}
\def\subsection{
\@startsection{subsection}{2}%
{\parindent}{.5\linespacing\@plus.7\linespacing}{-.5em}%
{\normalfont
}}

\begin{document}
\title{Blocks of affine quantum Schur algebras}
\author{Qiang Fu}
\address{Department of Mathematics, Tongji University, Shanghai, 200092, China.}
\email{q.fu@tongji.edu.cn}


\thanks{Supported by the National Natural Science Foundation
of China, the Program NCET, Fok Ying Tung Education Foundation
and the Fundamental Research Funds for the Central Universities}

\begin{abstract}
The affine quantum Schur algebra is a certain important infinite dimensional algebra whose representation theory is closely related to that of quantum affine $\frak{gl}_n$. Finite dimensional irreducible modules for the affine quantum Schur algebra $\afSrv$ were classified in  \cite{DDF}, where $v\in\mbc^*$ is not a root of unity.
We will classify blocks of the affine quantum Schur algebra $\afSrv$ in this paper.
\end{abstract}
 \sloppy \maketitle
\section{Introduction}
The classical Schur algebra $S(n,r)$ over an infinite field $F$  is a finite dimensional algebra whose module category is equivalent to the category of  $r$-homogeneous polynomial representations of the general linear group $GL_n(F)$.
The blocks of the Schur algebra $S(n,r)$ were determined by Donkin in \cite{Donkin} from the blocks of $GL_n(F)$. The $q$-Schur algebras are $q$-analogues of Schur algebras. When $q = 1$, these are the usual Schur algebras, and when $q$ is a prime power,
$q$-Schur algebras play an important role in the representation of finite general linear groups in the nondescribing characteristic case (see \cite{DJ89}).
It was proved by Cox \cite{Cox} that the blocks of $q$-Schur algebras $S_q(n,r)$ can be derived in the same way from the blocks of an appropriate quantum general linear group. The $q$-Schur algebras and quantum $\frak{gl}_n$ are related by quantum Schur--Weyl reciprocity  \cite{Jimbo,Du95,DPS} (see also \cite{ATY,SS,Ar,Hu} for the cyclotomic Schur--Weyl reciprocity).

The affine quantum Schur algebra is the affine version of the $q$-Schur algebra and it has several equivalent definitions (see \cite{GV,Gr99,Lu99,VV99}).
Unlike the $q$-Schur algebra, the affine quantum Schur algebra is an  infinite dimensional algebra.
Let $\afUglv$ be the quantized enveloping algebra of $\h{\frak{gl}}_n$ and let $\afSrv$ be the affine quantum Schur algebra over $\mbc$, where $v\in\mbc^*$ is not a root of unity.
It was proved in \cite[Th. 3.8.1]{DDF} that
there is a surjective algebra homomorphism $\zeta_r:\afUglv\ra\afSrv$ (cf. \cite{Fu,Po}).
Every simple $\afSrv$-module is inflated to a simple $\afUglv$-module
via $\zeta_r$. Furthermore each finite dimensional polynomial irreducible $\afUglv$-module can be regarded as an $\afSrv$-module via the map $\zeta_r$ for some $r$. Therefore, the representation theory of the affine quantum Schur algebras $\afSrv$ is closely related to that of $\afUglv$.

Finite dimensional irreducible modules for the quantum loop algebra $U_v(\h{\frak{g}})$ were classified by Chari--Pressley in \cite{CP91,CPbk,CP95}, where $\frak{g}$ is a complex finite dimensional simple Lie algebra. Furthermore the blocks of the category of finite dimensional $U_v(\h{\frak{g}})$-modules of type $1$ were classified in \cite{EM,CM}. Finite dimensional irreducible modules for the affine quantum Schur algebra $\afSrv$ were classified \cite[Th. 4.6.8]{DDF} in terms of Drinfeld polynomials.
Since the category $\hFnr$ of finite dimensional $\afSrv$-modules is not semisimple, the problem of determining the blocks in the category $\hFnr$  is very important. We will classify blocks of the category $\hFnr$ in Theorem \ref{block} and Theorem \ref{equivalent condition}.

It is well known that the blocks of extended affine Hecke algebras of type $A$ are classified by central characters (see \cite[Th. 7]{Muller}, \cite[III.9]{BG} and \cite[Th. 2.15]{LM}).
If $n\geq r$ then the category $\hFnr$ is equivalent to the category $\Cr$ of finite dimensional modules for the extended affine Hecke algebra $\afHrv$ of type $A$ (see \cite[Th. 4.2]{CP96} and \cite[Th. 4.1.3]{DDF}). Thus by Theorem \ref{block} we obtain a different approach to the classification of blocks in the category $\Cr$.

We organize this paper as follows. We recall the definition of quantum affine $\frak{gl}_n$ and the affine quantum Schur algebra $\afSrv$ in \S2. In \S3, we will recall some results about $\afUglv$ and $\afSrv$.
Let $\mba=\{\frac{f(u)}{g(u)}\in\mbc(u)\mid f(u),g(u)\in\mbc[u],\,f(0)=g(0)=1\}$ and let $\Xn=\mba^n$.
In \S4, we define the $\ell$-root lattice $\Rn$ of $\afUglv$ to be a certain subgroup of $\Xn$, and construct a certain  subset $\Xi_{n,r}$ of $\Xn/\Rn$. We will describe the set $\Xi_{n,r}$ in Proposition \ref{vartheta iso} and discuss the $\ell$-weights of polynomial irreducible $\afUglv$-modules in Corollary \ref{l-weight for L(bfQ)}. We will prove in Theorem \ref{block} that the set $\Xi_{n,r}$ is the index set of the blocks in the category of finite dimensional $\afSrv$-modules. In \S5, we will introduce the Weyl modules for the affine quantum Schur algebra $\afSrv$ and generalize Corollary \ref{l-weight for L(bfQ)} to the case of $\ell$-highest weight modules.
Blocks of the category of finite dimensional $\afSrv$-modules will be classified in Theorem  \ref{block} and Theorem \ref{equivalent condition}. As an application, we will use Theorem \ref{equivalent condition} to give a classification of the blocks for the affine Hecke algebra $\afHrv$ in Theorem \ref{block for affine Hecke algebras}.

\section{Quantum affine $\frak{gl}_n$ and affine quantum Schur algebras}

\subsection{}
Let $\ttv\in\mbc^*$ be a complex number which is not a root of unity,
where $\mbc^*=\mbc\backslash\{0\}$. Let $(c_{i,j})$ be the Cartan matrix of affine type $A_{n-1}$.

We recall the Drinfeld's new realization of quantum affine $\frak{gl}_n$ as follows (cf. \cite{FM}).
\begin{Def}\label{QLA}
The {\it quantum loop algebra} $\afUglv$ (or {\it
quantum affine $\mathfrak {gl}_n$}) is the $\mbc$-algebra generated by $\ttx^\pm_{i,s}$
($1\leq i<n$, $s\in\mbz$), $\ttk_i^{\pm1}$ and $\ttg_{i,t}$ ($1\leq
i\leq n$, $t\in\mbz\backslash\{0\}$) with the following relations:
\begin{itemize}
 \item[(QLA1)] $\ttk_i\ttk_i^{-1}=1=\ttk_i^{-1}\ttk_i,\,\;[\ttk_i,\ttk_j]=0$,
 \item[(QLA2)]
 $\ttk_i\ttx^\pm_{j,s}=\ttv^{\pm(\dt_{i,j}-\dt_{i,j+1})}\ttx^\pm_{j,s}\ttk_i,\;
               [\ttk_i,\ttg_{j,s}]=0$,
 \item[(QLA3)] $[\ttg_{i,s},\ttx^\pm_{j,t}]
               =\begin{cases}0,\;\;&\text{if $i\not=j,\,j+1$};\\
                  \pm \ttv^{-js}\frac{[s]_v}{s}\ttx^\pm_{j,s+t},\;\;\;&\text{if $i=j$};\\
                  \mp \ttv^{-js}\frac{[s]_v}{s}\ttx_{j,s+t}^\pm,\;\;\;&\text{if $i=j+1$,}
                \end{cases}$
 \item[(QLA4)] $[\ttg_{i,s},\ttg_{j,t}]=0$,
 \item[(QLA5)]
 $[\ttx_{i,s}^+,\ttx_{j,t}^-]=\dt_{i,j}\frac{\phi^+_{i,s+t}
 -\phi^-_{i,s+t}}{\ttv-\ttv^{-1}}$,
 \item[(QLA6)] $\ttx^\pm_{i,s}\ttx^\pm_{j,t}=\ttx^\pm_{j,t}\ttx^\pm_{i,s}$, for $|i-j|>1$, and
 $[\ttx_{i,s+1}^\pm,\ttx^\pm_{j,t}]_{\ttv^{\pm c_{ij}}}
               =-[\ttx_{j,t+1}^\pm,\ttx^\pm_{i,s}]_{\ttv^{\pm c_{ij}}}$,
 \item[(QLA7)]
 $[\ttx_{i,s}^\pm,[\ttx^\pm_{j,t},\ttx^\pm_{i,p}]_\ttv]_\ttv
 =-[\ttx_{i,p}^\pm,[\ttx^\pm_{j,t},\ttx^\pm_{i,s}]_\ttv]_\ttv\;$ for
 $|i-j|=1$,
\end{itemize}
 where $[x,y]_a=xy-ayx$, $[s]_v=\frac{v^s-v^{-s}}{v-v^{-1}}$ and $\phi_{i,s}^\pm$ are defined via the
 generating functions in indeterminate $u$ by
$$\Phi_i^\pm(u):={\ti\ttk}_i^{\pm 1}
\exp\bigl(\pm(\ttv-\ttv^{-1})\sum_{m\geq 1}\tth_{i,\pm m}u^{\pm
m}\bigr)=\sum_{s\geq 0}\phi_{i,\pm s}^\pm u^{\pm s}$$ with
$\ti\ttk_i=\ttk_i\ttk_{i+1}^{-1}$ ($\ttk_{n+1}=\ttk_1$) and $\tth_{i,\pm
m}=\ttv^{\pm(i-1)m}\ttg_{i,\pm m}-\ttv^{\pm(i+1)m}\ttg_{i+1,\pm
m}\,(1\leq i<n).$
\end{Def}

Let $\afUslv$ be
the subalgebra of $\afUglv$ generated by the elements $\ttx^\pm_{i,s}$, $\ti\ttk_i^{\pm1}$ and $\tth_{i,t}$ for $1\leq i\leq n-1$, $s\in\mbz$ and $t\in\mbz \backslash\{0\}$.
For $1\leq j<n$, let $E_j=\ttx^+_{j,0}$
and $F_j=\ttx^-_{j,0}$ and let
\begin{equation*}
\begin{split}
E_n&= \ttv[\ttx_{n-1,0}^-,[\ttx_{n-2,0}^-,\cdots,
[\ttx_{2,0}^-,\ttx_{1,1}^-]_{\ttv^{-1}}\cdots
]_{\ttv^{-1}}]_{\ttv^{-1}} \ti\ttk_n,\\
F_n&=\ttv^{-1}\ti\ttk_n^{-1}[\cdots[[\ttx_{1,-1}^+,\ttx_{2,0}^+]_\ttv,\ttx_{3,0}^+]_\ttv,
 \cdots,\ttx_{n-1,0}^+]_\ttv.
\end{split}
\end{equation*}
Then the algebra  $\afUslv$ is also generated by the elements $E_i$, $F_i$ and $\ti\ttk_i^{\pm 1}$ for $1\leq i\leq n$ (see \cite{Be}).
For $s\geq 1$ let
$$\sfz^\pm_s=
 s\ttv^{\pm s} \frac1{[s]_\ttv}(\ttg_{1,\pm s}+\cdots+\ttg_{n,\pm s}).$$
Then the elements $\sfz^\pm_s$ are central in $\afUglv$ and the algebra $\afUglv$ is generated by $E_i$, $F_i$, $\ttk_i^{\pm 1}$ and
$\sfz^\pm_s$ for $1\leq i\leq n$ and $s\geq 1$. Furthermore,
the algebra $\afUglv$ is a Hopf algebra with
comultiplication $\Dt:\afUglv\ra\afUglv\ot\afUglv$ defined by
\begin{equation}\label{Hopf}
\begin{split}
&\Delta(E_i)=E_i\otimes\ti \ttk_i+1\otimes
E_i,\quad\Delta(F_i)=F_i\otimes
1+\ti \ttk_i^{-1}\otimes F_i,\\
&\Delta(\ttk^{\pm 1}_i)=\ttk^{\pm 1}_i\otimes \ttk^{\pm 1}_i,\quad
\Delta(\sfz_s^\pm)=\sfz_s^\pm\otimes1+1\otimes
\sfz_s^\pm;\\
\end{split}
\end{equation}
where $1\leq i\leq n$ and $s\in \mbz^+$ (see \cite[2.2]{Hub2} and \cite[Cor. 2.3.5 and Prop. 4.4.1]{DDF}).

\subsection{}
We now recall the definition of affine quantum Schur algebras (see \cite{GV,Gr99,Lu99,VV99}).
Let $\Ogv$ be the $\mbc$-vector space with basis $\{\og_i\mid i\in\mathbb Z\}$.  The algebra $\afUglv$ acts on  $\Ogv$ by
\begin{equation*}
\aligned
E_i\cdot \og_s&=\dt_{\ol{i+1},\bar s}\og_{s-1},\quad F_i\cdot \og_s=\dt_{\bar i,\bar
s}\og_{s+1},\quad
\ttk_i^{\pm 1}\cdot \og_s=\ttv^{\pm\dt_{\bar i,\bar s}}\og_s,\\
&\sfz_t^+\cdot\og_s=\og_{s-tn},\quad\text{and }\;\;
\sfz_t^-\cdot\og_s=\og_{s+tn},
\endaligned
\end{equation*}
where $\bar i$ denotes the corresponding integer modulo $n$. The tensor space
$\Ogv^{\ot r}$ is a left $\afUglv$-module via the coproduct $\Delta$ on $\afUglv$.

The extended affine Hecke algebra $\afHrv$ of type $A$ is defined to be the $\mbc$-algebra generated by
$$T_i,\quad X_j^{\pm 1}(\text{$1\leq i\leq r-1$, $1\leq j\leq r$}),$$
 and relations
$$\aligned
 & (T_i+1)(T_i-\ttv^2)=0,\\
 & T_iT_{i+1}T_i=T_{i+1}T_iT_{i+1},\;\;T_iT_j=T_jT_i\;(|i-j|>1),\\
 & X_iX_i^{-1}=1=X_i^{-1}X_i,\;\; X_iX_j=X_jX_i,\\
 & T_iX_iT_i=\ttv^2 X_{i+1},\;\;  X_jT_i=T_iX_j\;(j\not=i,i+1).
\endaligned$$

Let $I(n,r)=\{(i_1,\ldots,i_r)\in\mbz^r\mid 1\leq i_k\leq
n,\,\forall k\}.$ We denote the symmetric group on $r$ letters by $\fSr$.
The symmetric group $\fSr$ acts
on the set $I(n,r)$ by place permutation:
\begin{equation*}\label{place permutation}
\bfi w=(i_{w(1)},\cdots,i_{w(r)}),\quad\text{
for $\bfi\in I(n,r)$ and $w\in\fSr$.}
\end{equation*}
For
 $\bfi=(i_1,\ldots,i_r)\in\mbz^r$, write
$$\og_\bfi=\og_{i_1}\ot\og_{i_2}\ot\cdots\ot \og_{i_r}=\og_{i_1}\og_{i_2}\cdots \og_{i_r}\in\Ogv^{\ot r}.$$
The algebra $\afHrv$ acts on  $\Ogv^{\ot r}$ on the right via
\begin{equation*}
\begin{split}
\og_{\bf i}\cdot T_k&=\left\{\begin{array}{ll} \ttv^2\og_{\bf
i},\;\;&\text{if $i_k=i_{k+1}$;}\\
\ttv\og_{\bfi s_k},\;\;&\text{if $i_k<i_{k+1}$;}\qquad\text{ for all }\bfi\in I(n,r),\\
\ttv\og_{\bfi s_k}+(\ttv^2-1)\og_{\bf i},\;\;&\text{if
$i_{k+1}<i_k$,}
\end{array}\right.\\
\og_{\bf i}\cdot X_t^{-1}
&=\og_{i_1}\cdots\og_{i_{t-1}}\og_{i_t+n}\og_{i_{t+1}}\cdots\og_{i_r},\qquad \text{ for all }\bfi\in \mbz^r;
\end{split}
\end{equation*}
for all $1\leq k\leq r-1$ and $1\le t\le r$, where $s_k:=(k,k+1)\in\fSr$  (see \cite{VV99}).
The endomorphism algebra $$\afSrv:=\End_{\afHrv}(\Ogv^{\ot r})$$
is called an affine quantum Schur algebra.

\subsection{}
Since the actions of $\afUglv$ and $\afHrv$ on $\Ogv^{\ot r}$ are commute (see
\cite[Prop. 3.5.5 and Prop. 4.4.1]{DDF}), there is an algebra homomorphism
\begin{equation}\label{zetar}
\zeta_r:\afUglv\ra\afSrv
\end{equation}
It was proved in \cite[Th. 3.8.1]{DDF} that $\zr$ is surjective.
Every $\afSrv$-module will be inflated into a $\afUglv$-module
via $\zeta_r$.

\section{Finite dimensional representation of $\afUglv$ and $\afSrv$}

\subsection{}
We first recall the classification of finite dimensional simple $\afUslv$-modules of type $1$.
Let $\mba=\{\frac{f(u)}{g(u)}\in\mbc(u)\mid f(u),g(u)\in\mbc[u],\,f(0)=g(0)=1\}$ and let $\Pn=\mba^{n-1}$. Then $\Pn$ is a group under multiplication.  For $1\leq i\leq n-1$ and $a\in\mbc^{*}$
let
\begin{equation}\label{ogia}
\og_{i,a}=(1,\ldots,1,\underset{(i)}{1-au},1,\ldots,1)\in\Pn.
\end{equation}
Then $\Pn$ is generated freely
as an abelian group by the elements $\og_{i,a}$, $1\leq i\leq n-1$, $a\in\mbc^*$. Let $\Pnp$ be the monoid generated  by $1$ and the elements $\og_{i,a}$, $1\leq i\leq n-1$, $a\in\mbc^*$.

For $1\leq j\leq n-1$ and $s\in\mbz$, define the elements $\ms
P_{j,s}\in\afUslv$ through the generating functions
\begin{equation}\label{ms Pj}
\begin{split}
& \ms P_j^\pm(u):=\exp\bigg(-\sum_{t\geq
1}\frac{1}{[t]_\ttv}\tth_{j,\pm t} (\ttv u)^{\pm
t}\bigg)=\sum_{s\geq 0}\ms P_{j,\pm s} u^{\pm
s}\in\afUslv[[u,u^{-1}]].
\end{split}
\end{equation}

For $g(u)=\prod_{1\leq i\leq m}(1-a_iu)\in\mbc[u]$
with constant term $1$ and $a_i\in\mbc^*$, define
\begin{equation}\label{f^pm(u)}
g^\pm(u)=\prod_{1\leq i\leq m}(1-a_i^{\pm1}u^{\pm1}).
\end{equation}
For
$\bfP=(P_1(u),\ldots,P_{n-1}(u))\in\Pnp$, define $P_{j,s}\in\mbc$,
for $1\leq j\leq n-1$ and $s\in\mbz$, by the following formula
\begin{equation}\label{Pis}
P_j^\pm(u)=\sum_{s\geq
0}P_{j,\pm s} u^{\pm s},
\end{equation}
where $P_j^\pm(u)$ is defined by
\eqref{f^pm(u)}.

For $\bfP\in\Pnp$ let $\Icp(\bfP)$ be the left ideal of $\afUslv$ generated by
$\ttx_{j,s}^+ ,\ms P_{j,s}-P_{j,s},$ and $\ti\ttk_j-\ttv^{\mu_j}$, for $1\leq j\leq n-1$ and $s\in\mbz$, where
$\mu_j=\mathrm{deg}P_j(u)$, and define
$$\Mcp(\bfP)=\afUslv/\Icp(\bfP).$$
Then $\Mcp(\bfP)$ has a unique simple quotient, denoted by
$\Lcp(\bfP)$.

The following result is due to Chari--Pressley (see
\cite{CP91,CPbk,CP95}).

\begin{Thm}
The modules $\Lcp(\bfP)$ with $\bfP\in\Pnp$ are all nonisomorphic
finite dimensional simple $\afUslv$-modules of  type $1$.
\end{Thm}

\subsection{}
Based on Chari--Pressley's classification of finite dimensional simple $\afUslv$-modules, finite dimensional polynomial irreducible $\afUglv$-modules were classified in \cite{FM}.
We now review the classification of irreducible polynomial representation of $\afUglv$.
Let
\begin{equation}\label{Xn}
\Xn=\mba^n=\Pno.
\end{equation}
For $1\leq i\leq n$ and $a\in\mbc^{*}$
let $\La_{i,a}\in\Xn$ be the element whose $i$th entry is $1-au$ and all other entries is $1$. In other words,
\begin{equation}\label{Laia}
\La_{i,a}=(1,\ldots,1,\underset{(i)}{1-au},1,\ldots,1)\in\Xn.
\end{equation}
Then $\Xn$ is generated freely
as an abelian group by the elements $\La_{i,a}$, $1\leq i\leq n$, $a\in\mbc^*$. Let $\Qn$ be the monoid generated  by $1$ and the elements $\La_{i,a}$, $1\leq i\leq n$, $a\in\mbc^*$.
Then by definition we have $\Qn=\Pnop$.

Following \cite{FM}, an $n$-tuple of polynomials
$\bfQ=(Q_1(u),\ldots,Q_n(u))$ with constant terms $1$ is called {\it
dominant} if, for each $1\leq i\leq n-1$, the ratio
$Q_i(v^{i-1}u)/Q_{i+1}(v^{i+1}u)$ is a polynomial. Let
$\Qnp$ be the set of dominant $n$-tuples of polynomials.

Let $\Uglv$ be the subalgebra of $\afUglv$ generated by all $\ttx^\pm_{i,0}$ and $\ttk_j^{\pm1}$ ($1\leq i\leq n-1$, $1\leq j\leq n$).
A finite dimensional representation $V$ of $\Uglv$ is said to be of type $1$
if $V=\oplus_{\la\in\mbzn}V_\la$, where
$V_\la=\{w\in V\mid \ttk_iw=v^{\la_i}w,\,1\leq i\leq n\}$. Let $$\wt(V)=\{\la\in\mbzn\mid V_\la\not=0\}.$$ If $\wt(V)\han\mbnn$ then $V$ is called a polynomial representation of $\Uglv$.

For $1\leq i\leq n$ and $s\in\mbz$, define the elements $\ms
Q_{i,s}\in\afUglv$ through the generating functions
\begin{equation}\label{ms Qj}
\begin{split}
&\quad\qquad\ms Q_i^\pm(u):=\exp\bigg(-\sum_{t\geq
1}\frac{1}{[t]_\ttv}g_{i,\pm t} (\ttv u)^{\pm t}\bigg)=\sum_{s\geq
0}\ms Q_{i,\pm s} u^{\pm s}\in\afUglv[[u,u^{-1}]].
\end{split}
\end{equation}

A $\afUglv$-module is said to be of type $1$ if $V$ is of type $1$ as a $\Uglv$-module.
Let $V$ be a finite dimensional $\afUglv$-module of type $1$. Then $V=\oplus_{\la\in\mbnn}V_\la$.
Since the elements $\ttk_i,\,\ms Q_{i,s}$ ($1\leq i\leq n$, $s\in\mbz$) commute among themselves, each $V_\la$ is a direct sum
of generalized eigenspaces of the form
\begin{equation*}\label{geigenspace}
V_{(\la,\gamma)}=\{x\in V_\la\mid (\ms
Q_{i,s}-\gamma_{i,s})^px=0\text{ for some $p$}\, (1\leq i\leq
n,s\in\mbz)\},
\end{equation*}
 where $\gamma=(\gamma_{i,s})$ with $\gamma_{i,s}\in\mbc$.

For $\bfQ=(Q_1(u),\ldots,Q_{n}(u))\in\Qn$, let
$$\deg\bfQ=(\deg Q_1(u),\cdots,\deg Q_n(u))\in\mbnn$$ and define
$Q_{i,s}\in\mbc$, for $1\leq i\leq n$ and $s\in\mbz$, by the
following formula
\begin{equation}\label{Qis}
Q_i^\pm(u)=\sum_{s\geq 0}Q_{i,\pm s}u^{\pm s},
\end{equation}
where $Q_i^\pm(u)$ is defined using \eqref{f^pm(u)}.
Given
$\bfQ\in\Qn$, let $$V_\bfQ=\{x\in V_\la\mid (\ms Q_{i,s}-Q_{i,s})^px=0\text{ for some $p$}\, (1\leq i\leq
n,s\in\mbz)\},$$ where $\la=\deg\bfQ$.

Following \cite{FM}, let $\hFn$ be the category of all finite dimensional representations $V$ of $\afUglv$ such that the restriction of $V$ to $\Uglv$ is a  polynomial representation of type $1$, and $V=\oplus_{\bfQ\in\Qn}V_\bfQ$. The objects of the category $\hFn$ are called polynomial representations of $\afUglv$.

For $\bfQ\in\Qnp$, let $I(\bfQ)$
be the left ideal of $\afUglv$ generated by $\ttx_{j,s}^+ ,\quad\ms
Q_{i,s}-Q_{i,s},$ and $\ttk_i-\ttv^{\la_i}$, for $1\leq j\leq n-1$,
$1\leq i\leq n$ and $s\in\mbz$, where $\la_i=\mathrm{deg}Q_i(u)$,
and define
$$M(\bfQ)=\afUglv/I(\bfQ).$$
Then $M(\bfQ)$ has a unique irreducible quotient, denoted by $L(\bfQ)$.
\begin{Thm}[\cite{FM}]\label{classification of simple afUglC-modules}
The modules $L(\bfQ)$ with $\bfQ\in\Qnp$ are all
nonisomorphic simple $\afUglv$-modules in the category $\hFn$.  Moreover,
\begin{equation}\label{res of L(Q)}
L(\bfQ)|_{\afUslv}\cong \bar L(\bfP),
\end{equation}
where
$\bfP=(P_1(u),\ldots,P_{n-1}(u))$ with
$P_i(u)=Q_i(\ttv^{i-1}u)/Q_{i+1}(\ttv^{i+1}u)$.
\end{Thm}

\subsection{}
Finally, we recall some results about affine quantum Schur algebras.
Let $\Qnpr=\{\bfQ\in\Qnp\mid\sum_{1\leq i\leq n}\deg Q_i(u)=r\}$.
\begin{Thm}\cite[Th. 4.6.8]{DDF}\label{classification of simple afSrv-modules}
For $\bfQ\in\Qnpr$ the $\afUglv$-module $L(\bfQ)$ can be regarded as an $\afSrv$-module via the map $\zeta_r$ defined in \eqref{zetar}. Furthermore, the set $\{L(\bfQ)\mid\bfQ\in\Qnpr\}$ is a complete set of nonisomorphic simple $\afSrv$-module.
\end{Thm}

\begin{Lem}\label{belong to hFn}
Let $V$ be a finite dimensional $\afUglv$-module of type $1$ and let $W$ be a submodule of $V$. If $W,V/W\in\hFn$ then $V\in\hFn$.
\end{Lem}
\begin{proof}
It is easy to see that $\dim V_{(\la,\ga)}=\dim W_{(\la,\ga)}+\dim (V/W)_{(\la,\ga)}$ for each generalized eigenvalue $(\la,\ga)$. In particular, we have $\dim V_\bfQ=\dim W_\bfQ+\dim (V/W)_\bfQ$
for $\bfQ\in\Qn$. This together with the fact that $W,V/W\in\hFn$, implies that $\dim V=\dim W+\dim V/W=\sum_{\bfQ\in\Qn}(\dim W_\bfQ+\dim (V/W)_\bfQ)=\sum_{\bfQ\in\Qn}V_\bfQ$. It follows that $V=\oplus_{\bfQ\in\Qn}V_\bfQ$ and hence $V\in\hFn$.
\end{proof}

Combining Theorem \ref{classification of simple afSrv-modules} with Lemma \ref{belong to hFn} yields the following result.

\begin{Coro}\label{finite dimensional afSrv modules}
Let $V$ be a finite dimensional $\afSrv$-module. Then $V\in\hFn$.
\end{Coro}

\section{The set $\Xi_{n,r}$ and the $\ell$-weights of $L(\bfQ)$}

\subsection{}

Following \cite[3.3]{CM} we introduce the $\ell$-simple roots of $\afUslv$ as follows.
For $1\leq i\leq n-1$ and $a\in\mbc^*$ let
$$\al_{i,a}=\og_{i-1,av}^{-1}\og_{i,a}\og_{i,av^2}\og_{i+1,av}^{-1}\in\Pn,$$
where $\og_{i,a}$ is defined in \eqref{ogia}. The elements $\al_{i,a}$ are called the $\ell$-simple roots for $\afUslv$.

Recall from \eqref{Xn} the definition of $\Xn$.
For $1\leq i\leq n-1$ and $a\in\mbc^*$ let
$$\bt_{i,a}=\La_{i,a}\La_{i+1,a}^{-1}\in\Xn,$$
where $\La_{i,a}$ is defined in \eqref{Laia}.
Let $\Rn$ be the subgroup of $\Xn$ generated by the elements $\bt_{i,a}$ ($1\leq i\leq n-1$, $a\in\mbc^*$). We call $\bt_{i,a}$ the $\ell$-simple roots of $\afUglv$ and $\Rn$ the $\ell$-root lattice of $\afUglv$.
Let $\Rnp$ be the monoid generated by $1$ and the elements $\bt_{i,a}$ ($1\leq i\leq n-1$, $a\in\mbc^*$), and $\Rnm=(\Rnp)^{-1}$.

The $\ell$-simple roots of $\afUglv$ and that of $\afUslv$ are related by a map $\kappa_v$, which we now describe. Clearly, we have
$$\mbz[\Xn]=\mbz[\La_{i,a}^{\pm 1}\mid 1\leq i\leq n,\, a\in\mbc^*],$$
where $\mbz[\Xn]$ is the group ring associated with $\Xn$.
There is a natural ring homomorphism $$\kav:\mbz[\Xn]\ra\mbz[\Pn]$$
defined by sending $\bfQ=(Q_1(u),\cdots,Q_n(u))$
to $\bfP=(P_1(u),\cdots,P_{n-1}(u))$, where $P_i(u)=Q_i(\ttv^{i-1}u)/Q_{i+1}(\ttv^{i+1}u)$ for $1\leq i\leq n-1$. By definition we have $$\kav(\La_{i,a})=\og_{i,av^{i-1}}\og_{i-1,av^i}^{-1}$$ for $1\leq i\leq n$, where $\og_{0,b}=\og_{n,b}=1$. It follows that
$$\kav(\bt_{i,a})=\al_{i,av^{i-1}},$$
for $1\leq i\leq n-1$ and $a\in\mbc^*$.

\subsection{}
For $n\geq 2$ and $r\in\mbn$ let
\begin{equation}\label{Xinr}
\begin{split}
\Xi_n&=\{\ol{\bfQ}\in\Xn/\Rn\mid\bfQ\in\Qnp\},\\ \Xi_{n,r}&=\{\ol{\bfQ}\in\Xn/\Rn\mid\bfQ\in\Qnpr\}.
\end{split}
\end{equation}
We will prove in Theorem \ref{block} that $\Xi_n$ is the index set of the blocks in the category $\hFn$ and $\Xi_{n,r}$ is  the index set of the blocks in the category of finite dimensional $\afSrv$-modules.

We are now prepared to describe the sets $\Xi_n$ and $\Xi_{n,r}$ in Propostion \ref{vartheta iso} below.
For $m\leq n-1$, let $\sX_m$ be the subgroup of $\Xn$ generated by $\La_{i,a}$ for $1\leq i\leq m$ and $a\in\mbc^*$.
For $m\leq s\leq n-1$ let $\sY_{m,s}$ be the subgroup of $\Xn$ generated by $\bt_{i,a}$ for $m\leq i\leq s$ and $a\in\mbc^*$.

\begin{Lem}\label{Xm cap Yms}
For $m\leq s\leq n-1$ we have
$\sX_m\cap \sY_{m,s}=\{{\bf 1}\}$, where ${\bf 1}=(1,\cdots,1)$ is the identity element in $\Xn$.
\end{Lem}
\begin{proof}
Assume $m+1\leq s\leq n-1$. Let $x\in \sX_m\cap \sY_{m,s}$.
Since $\sY_{m,s}=\sY_{m,s-1}\sY_{s,s}$, there exist $y\in\sY_{m,s-1}$, $a_1,\cdots,a_t\in\mbc^*$ and $k_1,\cdots,k_t\in\mbz$ such that
$$x=y\cdot\bt_{s,a_1}^{k_1}\cdots\bt_{s,a_t}^{k_t}=y\cdot\La^{k_1}_{s,a_1}\cdots
\La_{s,a_t}^{k_t}\La_{s+1,a_1}^{-k_1}\cdots\La_{s+1,a_t}^{-k_t}.$$
Since $y\in \sY_{m,s-1}\han \sX_s$ and $x\in \sX_m\han \sX_s$ we conclude that
$$\La_{s+1,a_1}^{k_1}\cdots\La_{s+1,a_t}^{k_t}=y\cdot\La^{k_1}_{s,a_1}\cdots
\La_{s,a_t}^{k_t}\cdot x^{-1}\in \sX_s.$$ This, together with the fact that
$\Xn$ is generated freely
as an abelian group by the elements $\La_{i,a}$ ($1\leq i\leq n$, $a\in\mbc^*$), implies that $k_i=0$ for $1\leq i\leq t$. It follows that $x=y\in \sY_{m,s-1}$. This shows that $\sX_m\cap \sY_{m,s}\han \sX_m\cap \sY_{m,s-1}$ for $m+1\leq s\leq n-1$. Thus, $\sX_m\cap \sY_{m,s}\han \sX_m\cap \sY_{m,m}$ for $m\leq s\leq n-1$.
Furthermore, since $\Xn$ is generated freely
as an abelian group by the elements $\La_{i,a}$, $1\leq i\leq n$, $a\in\mbc^*$, we conclude that
$\sX_m\cap \sY_{m,m}=\{{\bf 1}\}.$ Consequently, we have $\sX_m\cap \sY_{m,s}=\{{\bf 1}\}$ for $m\leq s\leq n-1$.
\end{proof}

Recall from \S3 that $\mba=\{\frac{f(u)}{g(u)}\in\mbc(u)\mid f(u),g(u)\in\mbc[u],\,f(0)=g(0)=1\}$
and $\mba$ is a group under multiplication. There is a natural group homomorphism
\begin{equation}\label{vartheta}
\vartheta:\mba\ra\Xn/\Rn
\end{equation}
defined by sending $Q(u)$ to $\ol{(Q(u),1,\cdots,1)}$.
Let
\begin{equation*}
\begin{split}
\Ga&=\{Q(u)\in\mbc[u]\mid Q(0)=1\},\\
\Ga_r&=\{ Q(u)\in\Ga\mid\deg(Q(u))=r\}.
\end{split}
\end{equation*}

\begin{Prop}\label{vartheta iso}
(1) The map $\vartheta$ defined in \eqref{vartheta} is a group isomorphism. Thus the group $\Xn/\Rn$ is a free abelian group and the set $\{\ol{\La_{1,a}}\mid a\in\mbc^*\}$ is a basis of $\Xn/\Rn$.

(2) We have $\vartheta(\Ga)=\Xi_n$ and $\vartheta(\Ga_r)=\Xi_{n,r}$ for all $r$.

(3) Let $\bfQ=(Q_1(u),\cdots,Q_n(u))\in\Qnpr$ and $\bfQ'=(Q_1'(u),\cdots,Q_n'(u))\in\Qnpt$.
Then $\ol{\bfQ}=\ol{\bfQ'}$ if and only if $r=t$ and $\prod_{1\leq i\leq n}Q_i(u)=\prod_{1\leq i\leq n}Q_i'(u)$.
\end{Prop}
\begin{proof}
Since $\ol{\La_{i,a}}=\ol{\La_{i+1,a}}$ for $1\leq i\leq n-1$,  the group $\Xn/\Rn$ is generated by the elements $\ol{\La_{1,a}}$ for $a\in\mbc^*$. This implies that $\vartheta$ is surjective. Let $Q(u)=(1-a_1u)^{m_1}\cdots (1-a_tu)^{m_t}\in\ker\vartheta$, where $a_i\in\mbc^*$ and $m_i\in\mbz$. Then by Lemma \ref{Xm cap Yms} we have $\La_{1,a_1}^{m_1}\cdots\La_{1,a_t}^{m_t}\in\sX_1\cap\Rn=\sX_1\cap \sY_{1,n-1}=\{{\bf1}\}$. This implies that $m_i=0$ for all $i$ and hence $\vartheta$ is injective. The assertion (1) follows.

It is clear that $\vartheta(\Ga)\han\Xi_{n}$ and $\vartheta(\Ga_r)\han\Xi_{n,r}$ for all $r$. On the other hand, for   $\bfQ=(Q_1(u),\cdots,Q_n(u))\in\Qnpr$, let
$Q(u)=\prod_{1\leq i\leq n}Q_i(u)$. Then $Q(u)\in\Ga_r$.
For $1\leq i\leq n$, we write
$$Q_i(u)=\prod_{k_{i-1}+1\leq s\leq k_i}(1-a_{s}u),
$$
where $k_{i}=\sum_{1\leq s\leq i}\deg(Q_i(u))$ and $k_0=1$.
Since $\ol{\La_{1,a}}=\ol{\La_{i,a}}$ for $1\leq i\leq n$ we conclude that
\begin{equation}\label{barQ}
\ol{\bfQ}=\prod_{1\leq i\leq n\atop k_{i-1}+1\leq s\leq k_i}\ol{\La_{i,a_s}}=\prod_{1\leq s\leq r}\ol{\La_{1,a_s}}=\vartheta(Q(u))\in\vartheta(\Ga_r),
\end{equation}
and hence the assertion (2) follows.
The assertion (3) follows from \eqref{barQ} and (1).
\end{proof}

\subsection{}
For $V\in\hFn$ let $$\wtl(V)=\{\bfQ\in\Qn\mid V_\bfQ\not=0\}.$$
The element $\bfQ$ in $\wtl(V)$ is called an $\ell$-weight of $V$. We will prove in Corollary \ref{l-weight for L(bfQ)} that
$\wtl(L(\bfQ))\han\bfQ\Rnm$ for $\bfQ\in\Qnp$.

Following \cite{FM}, for $V\in\hFn$, the $v$-character of $V$ is defined as follows
$$\ch (V)=\sum_{\bfQ\in\Qn}\dim V_\bfQ \bfQ\in \mbz[\La_{i,a}\mid 1\leq i\leq n,\, a\in\mbc^*]\han\mbz[\Xn].$$
For $1\leq i\leq n$ and $a\in\mbc^*$, define
$\bfQ_{i,a}=(Q_1(u),\cdots,Q_n(u))\in\Qnp$ by setting $Q_n(u)=(1-av^{-n+1}u)^{\dt_{i,n}}$ and
$$\frac{Q_j(u\ttv^{j-1})}
{Q_{j+1}(u\ttv^{j+1})}=(1-au)^{\dt_{i,j}}$$ for $1\leq j\leq n-1$.
Then $\Qnp$ is generated by the elements $\bfQ_{i,a}$ ($1\leq i\leq n$, $a\in\mbc^*$) as a monoid.
According to \cite[(4.10)]{FM} we have
\begin{equation}\label{ch(L(Qia))}
\ch (L(\bfQ_{i,a}))=\sum_{1\leq j_1<\cdots<j_i\leq n}\La_{j_1,av^{i-1}}\La_{j_2,av^{i-3}}\cdots\La_{j_i,av^{1-i}}.
\end{equation}

For $r\geq 0$  we set
$$\Lanr=\big\{\la\in\mbnn\mid\sum_{1\leq i\leq n}\la_i=r\big\}.$$ Let $\La^+(n,r)=\{\la\in\Lanr\mid\la_i\geq\la_{i+1},\,1\leq i\leq n-1\}$. For $\la,\mu\in\Lanr$ write
$\la\unlhd\mu$ if $\sum_{j=1}^i\la_j\leq\sum_{j=1}^i\mu_j$ for $1\leq i\leq n$.

\begin{Lem}\label{l-weight for L(Qia)}
For $1\leq i\leq n$ and $a\in\mbc^*$ we have
$\wtl(L(\bfQ_{i,a}))\han\bfQ_{i,a}\Rnm$.
\end{Lem}
\begin{proof}
Let $\sM_i=\{\bfj=(j_1,\cdots,j_i)\in\mbnn\mid 1\leq j_1<\cdots<j_i\leq n\}$.
For $\bfj\in\sM_i$ and $a\in\mbc^*$ let
where $$\La_{\bfj,a}=\La_{j_1,av^{i-1}}\La_{j_2,av^{i-3}}\cdots\La_{j_i,av^{1-i}}.$$
According to \eqref{ch(L(Qia))}, we have $\wtl(L(\bfQ_{i,a}))=\{\La_{\bfj,a}\mid\bfj\in\sM_i\}$. Thus
we have to show that $\La_{\bfj,a}\in
\bfQ_{i,a}\Rnm$ for $\bfj\in\sM_i$.  For $1\leq i\leq n$ let
\begin{equation}\label{ei}
\bse_i=(0,\cdots,0,\underset
{(i)}1,0,\cdots,0)\in\La(n,1).
\end{equation}
For $\bfj,\bfj'\in\sM_i$ we write $$\bfj\geq\bfj' \Leftrightarrow \bse_{j_1}+\bse_{j_2}+\cdots+\bse_{j_i}\unlhd\bse_{j_1'}+\bse_{j_2'}+\cdots+\bse_{j_i'}.$$ Then $(1,2,\cdots,i,0,\cdots,0)\leq\bfj$ for $\bfj\in\sM_i$.
We proceed by induction on the ordering $\leq$ on $\sM_i$. Clearly, we have $\La_{(1,2,\cdots,i,0,\cdots,0),a}=\bfQ_{i,a}\in\bfQ_{i,a}\Rnm$. Assume now that
$\bfj\in\sM_i$ and $\bfj\not=(1,2,\cdots,i,0,\cdots,0)$. Let $\la=\bse_{j_1}+\bse_{j_2}+\cdots+\bse_{j_i}$. Then $\la\not\in\La^+(n,i)$ and hence there exists $1\leq k\leq n-1$ such that $0\leq\la_k<\la_{k+1}\leq 1$. It follows that $\la_k=0$ and $\la_{k+1}=1$. Since $\la=\bse_{j_1}+\bse_{j_2}+\cdots+\bse_{j_i}$ and $\la_{k+1}=1$, there exists $1\leq s\leq i$ such that $j_s=k+1$. Let $\bfj'=(j_1,\cdots,j_{s-1},j_s-1,j_{s+1},\cdots,j_i)$.
Since $\la_{j_{s-1}}=1$ and $\la_{j_s-1}=\la_k=0$ we have $j_s-1\not=j_{s-1}$. This together with the fact that $j_{s-1}<j_s$, implies that $j_{s-1}<j_s-1$. Consequently, $\bfj'\in\sM_i$.
Let $\mu=\bse_{j_1'}+\bse_{j_2'}+\cdots+\bse_{j_i'}$. Then $\mu=\la+\bse_k-\bse_{k+1} \rhd\la$ and hence $\bfj>\bfj'$. By induction we have $\La_{\bfj',a}\in\bfQ_{i,a}\Rnm$. Furthermore we have  $$\La_{\bfj,a}=\La_{\bfj',a}\bt_{j_s-1,av^{-2s+i+1}}^{-1}
=\La_{\bfj',a}\bt_{k,av^{-2s+i+1}}^{-1}.$$ Thus we conclude that $\La_{\bfj,a}\in\bfQ_{i,a}\Rnm$. The proof is completed.
\end{proof}

\begin{Coro}\label{l-weight for L(bfQ)}
For $\bfQ\in\Qnp$ we have $\wtl(L(\bfQ))\han\bfQ\Rnm$.
\end{Coro}
\begin{proof}
According to the proof of \cite[Prop. 4.6.7]{DDF} we know that
there exist $1\leq i_1,\cdots,i_r\leq n$ and $a_1,\cdots,a_r\in\mbc^*$ such that $\bfQ=\bfQ_{i_1,a_1}\cdots\bfQ_{i_r,a_r}$ and $L(\bfQ)$ is a subquotient module of $L(\bfQ_{i_1,a_1})\ot L(\bfQ_{i_2,a_2})\ot\cdots\ot L(\bfQ_{i_r,a_r})$.  Thus by \cite[Lem. 4.1]{FM} and \ref{l-weight for L(Qia)} we conclude that
$\wtl(L(\bfQ))\han\wtl(L(\bfQ_{i_1,a_1}))\wtl(L(\bfQ_{i_2,a_2}))
\cdots\wtl(L(\bfQ_{i_r,a_r}))\han\bfQ_{i_1,a_1}\cdots\bfQ_{i_r,a_r}\Rnm=\bfQ\Rnm$.
\end{proof}


\section{Weyl modules for affine quantum Schur algebras}

\subsection{}
For $\bfP\in\Pnp$ let $\Jcp(\bfP)$ be the left ideal of $\afUslv$ generated by $\ttx_{j,s}^+ ,\ms P_{j,s}-P_{j,s},$  $\ti\ttk_j-v^{\mu_j}$, and $(\ttx_{j,s}^-)^{\mu_j+1}$ for $1\leq j\leq n-1$ and $s\in\mbz$, where $\mu_j=\deg(P_j(u))$, $\ms P_{j,s}$ is defined in \eqref{ms Pj} and $P_{j,s}$ is defined using \eqref{Pis}. For $\bfP\in\Pnp$ let
$$\Wcp(\bfP)=\afUslv/\Jcp(\bfP).$$
The modules $\Wcp(\bfP)$ are called Weyl modules for $\afUslv$. It was proved in  \cite{CP01} that for each $\bfP\in\Pnp$, $\Wcp(\bfP)$ is a finite dimensional $\afUslv$-module of type $1$.

\begin{Thm}[\cite{CM}]\label{Decomposition of Weyl modules for affine sln}
For $\bfP\in\Pnp$ there exist $1\leq i_1,i_2,\cdots,i_k\leq n-1$ and $a_1,a_2,\cdots,a_k\in\mbc^*$ such that $\Wcp(\bfP)\cong\Lcp(\og_{i_1,a_1})\ot
\Lcp(\og_{i_2,a_2})\ot\cdots
\ot \Lcp(\og_{i_k,a_k})$.
\end{Thm}

\subsection{}
For $\bfQ\in\Qnp$ let $J(\bfQ)$ be the left ideal of $\afUglv$ generated by $\ttx_{j,s}^+ ,\ms Q_{i,s}-Q_{i,s},$  $\ttk_i-v^{\la_i}$, and $(\ttx_{j,s}^-)^{\mu_j+1}$ for $1\leq j\leq n-1$,
$1\leq i\leq n$, and $s\in\mbz$, where $\la_i=\mathrm{deg}Q_i(u)$, $\mu_j=\la_j-\la_{j+1}$, $\ms Q_{i,s}$ is defined in \eqref{ms Qj} and $Q_{i,s}$ is defined using \eqref{Qis}. For $\bfQ\in\Qnp$ let
$$W(\bfQ)=\afUglv/J(\bfQ).$$
The modules $W(\bfQ)$ are called Weyl modules for $\afUglv$.

A $\afUglv$-module $V$ in $\hFn$ is called an $\ell$-highest weight module with $\ell$-highest weight $\bfQ\in\Qnp$ if there exists a nonzero vector $w_0\in V$ such that $V=\afUglv w_0$ and
$$\ttx_{j,s}^+w_0=0,\quad\ms Q_i^\pm(u)w_0=Q_i^\pm(u)w_0,\quad\ttk_iw_0=v^{\la_i}w_0$$
for $1\leq j\leq n-1$, $1\leq i\leq n$ and $s\in\mbz$, where $\la_i=\deg(Q_i(u))$, $\ms Q_i^\pm(u)$ is defined in \eqref{ms Qj} and $Q_i^\pm(u)$ is defined by \eqref{f^pm(u)}. The element $w_0$ is called the $\ell$-highest weight vector of $V$.

\begin{Lem}\label{l-highest weight module}
Any $\ell$-highest weight $\afUglv$-module in $\hFn$ is a quotient of $W(\bfQ)$
for some $\bfQ\in\Qnp$.
\end{Lem}
\begin{proof}
Let $V$ be an $\ell$-highest weight $\afUglv$-module in $\hFn$ with $\ell$-highest weight $\bfQ\in\Qnp$ and let $w_0$ be the $\ell$-highest weight vector of $V$.
Then by definition we have $V=\afUglv w_0$ and
$$\ttx_{j,s}^+w_0=0,\quad\ms Q_i^\pm(u)w_0=Q_i^\pm(u)w_0,\quad\ttk_iw_0=v^{\la_i}w_0$$
for $1\leq j\leq n-1$, $1\leq i\leq n$ and $s\in\mbz$, where $\la_i=\deg(Q_i(u))$.
Since $\dim V<\infty$, by \cite[Lem. 5.4]{Jantzen} we conclude that $(x_{j,s}^-)^{\mu_j+1}w_0=0$ for $1\leq j\leq n-1$ and $s\in\mbz$, where $\mu_j=\la_j-\la_{j+1}$. It follows that $J(\bfQ)w_0=0$ and hence $V$ is a quotient of $W(\bfQ)$.
\end{proof}

\subsection{}
We are now ready to extend Theorem \ref{Decomposition of Weyl modules for affine sln} to the case of $\afUglv$. We need some preliminary lemmas.
\begin{Lem}\label{pseudo-highest afUglv-module}
Let $V$ be an $\ell$-highest $\afUglv$-weight module in $\hFn$ with $\ell$-highest weight $\bfQ$ and let $w_0$ be the $\ell$-highest weight vector of $V$. Then $V=\afUslv w_0$.
\end{Lem}
\begin{proof}
Let $\la=\deg(\bfQ)=(\deg Q_1(u),\cdots,\deg Q_n(u))$. Then $w_0\in V_\la$ and $\dim V_\la=1$.
Since the elements $\sfz^\pm_s$ are central in $\afUglv$, we have
$\sfz^\pm_sw_0\in V_\la$. Thus, since $\dim V_\la=1$, there exists
$a_{\pm s}\in\mbc$ such that
$\sfz^\pm_sw_0=a_{\pm s}w_0$. This together with the fact that $\afUglv$ is
generated by $\afUslv$, $\ttk_i^{\pm 1}$ and $\sfz^\pm_s$ ($1\leq i\leq n$, $s\geq 1$), implies that $V=\afUglv w_0=\afUslv w_0$.
\end{proof}

For $a\in\mbc^*$ let $$\Det_a=L(\bfQ_{n,av^{n-1}}).$$
According to \cite[Lem. 4.6.5]{DDF} we have $\dim\Det_a=1$ for $a\in\mbc^*$.
\begin{Lem}\label{Prop of Det}
Let $a\in\mbc^*$ and $V$ be a $\afUglv$-module.
Assume  $\Det_a=\spann\{w_0\}$. Then $$x(w\ot w_0)=(xw)\ot w_0$$ for $x\in\afUslv$ and $w\ot w_0\in V\ot\Det_a$. In particular we have $(V\ot\Det_a)|_{\afUslv}\cong V|_{\afUslv}$.
\end{Lem}
\begin{proof}
It is clear that $E_iw_0=F_iw_0=0$ and $\ti\ttk_iw_0=w_0$ for all $i$. This together with \eqref{Hopf} implies that
$E_i(w\ot w_0)=(E_iw)\ot w_0$,  $F_i(w\ot w_0)=(F_iw)\ot w_0$ and
$\ti\ttk_i(w\ot w_0)=(\ti\ttk_i w)\ot w_0$ for $1\leq i\leq n$. Therefore, since $\afUslv$ is generated by the elements $E_i$, $F_i$ and $\ti\ttk_i^{\pm 1}$ for $1\leq i\leq n$, we conclude that $x(w\ot w_0)=(xw)\ot w_0$ for $x\in\afUslv$ and $w\in V$. The proof is completed.
\end{proof}
\begin{Prop}\label{Decomposition of Weyl modules for affine gln}
(1) For  $\bfQ\in\Qnp$, we have $W(\bfQ)|_{\afUslv}\cong\Wcp(\bfP)$
where $\bfP\in\Pnp$ with $P_i(u)=Q_i(\ttv^{i-1}u)/Q_{i+1}(\ttv^{i+1}u)$ for $1\leq i\leq n-1$.

(2) For each $\bfQ\in\Qnp$, there exist $1\leq i_1, \cdots,i_k\leq n-1$, $a_1, \cdots,a_k\in\mbc^*$ and $b_1,\cdots,b_l\in\mbc^*$ such that  $W(\bfQ)\cong L(\bfQ_{i_1,a_1})\ot\cdots
\ot L(\bfQ_{i_k,a_k})\ot \Det_{b_1}\ot\cdots\ot\Det_{b_l}$.

(3) For $\bfQ\in\Qnp$ we have $W(\bfQ)\in\hFn$.
\end{Prop}
\begin{proof}
There exists a natural $\afUslv$-module homomorphism
\begin{equation}\label{vi}
\vi:\Wcp(\bfP)\ra W(\bfQ)
\end{equation}
defined by sending $x+\Jcp(\bfP)$ to $x+J(\bfQ)$ for $x\in\afUslv$.
From Lemma \ref{pseudo-highest afUglv-module} we see that $W(\bfQ)=\afUslv \bar 1$, where $\bar 1=1+J(\bfQ)\in W(\bfQ)$. This implies that $\vi$ is surjective. By \cite{CP01} we know that $\Wcp(\bfP)$ is finite dimensional. Thus $\dim W(\bfQ)<\infty$ and
\begin{equation}\label{geq}
 \dim \Wcp(\bfP)\geq\dim W(\bfQ).
\end{equation}
To prove (1) it will therefore be enough to prove that
$\dim W(\bfQ)\geq\dim \Wcp(\bfP)$.

According to Theorem \ref{Decomposition of Weyl modules for affine sln} and \eqref{res of L(Q)}, there exist $1\leq i_1,i_2,\cdots,i_k\leq n-1$ and $a_1,a_2,\cdots,a_k\in\mbc^*$ such that
\begin{equation}\label{iso}
\Wcp(\bfP)\cong\Lcp(\og_{i_1,a_1})\ot\cdots
\ot \Lcp(\og_{i_k,a_k})\cong (L(\bfQ_{i_1,a_1})\ot\cdots\ot L(\bfQ_{i_k,a_k}))|_{\afUslv}.
\end{equation}
Let $\la=\deg\bfQ=(\deg Q_1(u),\cdots,\deg Q_n(u))$.
Write $Q_n(u)=(1-b_1u)\cdots(1-b_lu)$ where $l=\la_n=\deg(Q_n(u))$. Since $\bfP=\og_{i_1,a_1}\cdots\og_{i_k,a_k}$ we have $\bfQ=\bfQ_{i_1,a_1}\cdots\bfQ_{i_k,a_k}\bfQ_{n,b_1v^{n-1}}\cdots\bfQ_{n,b_lv^{n-1}}$.
Let $$V=L(\bfQ_{i_1,a_1})\ot\cdots
\ot L(\bfQ_{i_k,a_k})\ot \Det_{b_1}\ot\cdots\ot\Det_{b_l}.$$
Then $\dim V_\la=1$. Let $w_{s}$ (resp., $m_t$) be an $\ell$-highest weight vector of $L(\bfQ_{i_s,a_s})$ (resp., $\Det_{b_t}$). Let $w_0=w_1\ot\cdots\ot w_k$ and $m_0=m_1\ot\cdots\ot m_l$.
 From Lemma \ref{Prop of Det} and \eqref{iso} we see that
\begin{equation}\label{res}
V|_{\afUslv}\cong \Wcp(\bfP)
\end{equation}
It follows that $\Jcp(\bfP)(w_0\ot m_0)=0$. Furthermore by \cite[Lem. 4.1]{FM} we have
$\ms Q_i^\pm(u)(w_0\ot m_0)=Q_i^\pm(u)w_0$ and $\ttk_i(w_0\ot m_0)=v^{\la_i}w_0\ot m_0$
for $1\leq i\leq n$. Thus there exists a $\afUglv$-module homomorphism
$$\psi:W(\bfQ)\ra V $$
defined by sending $x\bar 1$ to $x(w_0\ot m_0)$ for $x\in\afUglv$, where $\bar 1=1+J(\bfQ)\in W(\bfQ)$.
According to \eqref{iso} we have
$L(\bfQ_{i_1,a_1})\ot\cdots\ot L(\bfQ_{i_k,a_k})=\afUslv w_0$. This together with Lemma
\ref{Prop of Det} implies that
$$V\supseteq\afUglv(w_0\ot m_0)\supseteq\afUslv(w_0\ot m_0)=\{(xw_0)\ot m_0\mid
x\in\afUslv\}=V.$$
It follows that $V=\afUglv(w_0\ot m_0)$ and hence $\psi$ is surjective. Therefore by  \eqref{geq}, \eqref{iso} and \eqref{res} we have
$$\dim W(\bfQ)\geq \dim V=\dim\Wcp(\bfP)\geq\dim W(\bfQ).$$
Thus, the maps $\vi$, $\psi$ are all isomorphisms. The assertion (1) and (2) follows.
The assertion (3) follows from (2) and \cite[Lem. 4.3]{FM}.
The proof is completed.
\end{proof}

\subsection{}
We will prove in Corollary \ref{Weyl module for affine q-Schur algebras} that the $\afUglv$-module $W(\bfQ)$ is also an $\afSrv$-module for $\bfQ\in\Qnpr$. The $\afSrv$-modules $W(\bfQ)$ with $\bfQ\in\Qnpr$ will be called Weyl modules for $\afSrv$.

According to Theorem \ref{classification of simple afSrv-modules} and \cite[Lem. 4.6.6]{DDF} we have the following result.
\begin{Lem}\label{tensor of simple module}
Let $\bfQ_i\in\Qnpri$ ($1\leq i\leq t$). Then
$L(\bfQ_1)\ot\cdots\ot L(\bfQ_t)$ can be regarded as an $\afSrv$-module via the map $\zeta_r$ defined in \eqref{zetar}, where $r=r_1+\cdots+r_t$.
\end{Lem}

Combining Theorem \ref{classification of simple afSrv-modules} with \cite[(4.6.0.1) and Cor. 4.6.2]{DDF} gives the following result.
\begin{Lem}\label{Decomposition of L(bfQ)}
For $\bfQ\in\Qnpr$ there exist $a_1,\cdots,a_r\in\mbc^*$ such that
$L(\bfQ)$ is a quotient of $L(\bfQ_{1,a_1})\ot\cdots\ot L(\bfQ_{1,a_r})$.
\end{Lem}

\begin{Coro}\label{Weyl module for affine q-Schur algebras}
Let $\bfQ\in\Qnpr$. Then the $\afUglv$-module $W(\bfQ)$ can be regarded as an $\afSrv$-module via the map $\zeta_r$ defined in \eqref{zetar}. Furthermore,
there exist $a_1,\cdots,a_r\in\mbc^*$ such that $W(\bfQ)$ is a  quotient of $L(\bfQ_{1,a_1})\ot\cdots\ot L(\bfQ_{1,a_r})$. \end{Coro}
\begin{proof}
According to Proposition \ref{Decomposition of Weyl modules for affine gln}(2) and Lemma \ref{Decomposition of L(bfQ)}, we conclude that there exist $a_1,\cdots,a_{r'}$ such that $W(\bfQ)$ is a quotient of $L(\bfQ_{1,a_1})\ot\cdots\ot L(\bfQ_{1,a_{r'}})$. We have to show that $r=r'$. From Lemma \ref{tensor of simple module} we see that $L(\bfQ_{1,a_1})\ot\cdots\ot L(\bfQ_{1,a_{r'}})$ is an $\afSrprimev$-module. It implies that $W(\bfQ)$ is also an $\afSrprimev$-module.
This, together with the fact that  $\zeta_{r'}(\ttk_1\cdots \ttk_n)=v^{r'}$ shows that
$$v^{r'}w_0=\zeta_{r'}(\ttk_1\cdots \ttk_n)w_0=\ttk_1\cdots \ttk_nw_0=v^{\la_1+\cdots+\la_n}w_0$$
where $w_0$ is the $\ell$-highest weight vector of $W(\bfQ)$ and $\la_i=\deg Q_i(u)$. Therefore $r'=\la_1+\cdots+\la_n=r$.
\end{proof}
\subsection{}
We now use Proposition \ref{Decomposition of Weyl modules for affine gln} to prove the following generalization of Corollary \ref{l-weight for L(bfQ)}.

\begin{Prop}\label{l-weight for l highest weight module}
Let $V\in\hFn$ be an $\ell$-highest weight $\afUglv$-module with $\ell$-highest weight $\bfQ\in\Qnp$. Then $\wtl(V)\han\bfQ\Rnm$.
\end{Prop}
\begin{proof}
According to \cite[Lem. 4.1]{FM}, Lemma \ref{l-weight for L(Qia)} and Proposition \ref{Decomposition of Weyl modules for affine gln}(2) we have $\wtl(W(\bfQ))\han\bfQ\Rnm$. Thus by Lemma \ref{l-highest weight module} we conclude that $\wtl(V)\han\bfQ\Rnm$.
\end{proof}

\section{Blocks of $\afUglv$ and $\afSrv$}

\subsection{}
Let us recall the definition of blocks of an abelian category as follows (see \cite{EM}). Let $\sC$ be an abelian category, in which every object has finite length. Say that two indecomposable objects $X_1,X_2$ of $\sC$ are linked if there do not exist abelian subcategories $\sC_k$ ($1\leq k\leq 2$) such that $\sC =\sC_1\oplus \sC_2$ with
$X_1 \in \sC_1$, $X_2\in\sC_2$. Then linking is an equivalence relation. Let $I$ be the set of equivalence classes of linked indecomposable objects in $\sC$. For $\al\in I$ let $\sC_\al$ be the subcategory of $\sC$, consisting of direct sums of indecomposable objects
from $\al$. Then $\sC=\oplus_{\al\in I}\sC_\al$. The subcategories $\sC_\al$ are called the blocks of $\sC$.

For $\bfQ\in\Qnp$ the element $\chi_\bfQ:=\ol\bfQ\in\Xn/\Rn$ is called the elliptic character of $L(\bfQ)$. A $\afUglv$-module $V$ in $\hFn$ is said to have elliptic character $\chi\in\Xn/\Rn$ if every irreducible composition factor of $V$ has elliptic character $\chi$. Let $\hFnchi$ be the subcategory of $\hFn$ consisting of $\afUglv$-modules in $\hFn$ with elliptic character $\chi$.

Let $\hFnr$ be the category of finite dimensional $\afSrv$-modules. According to Corollary \ref{finite dimensional afSrv modules} we know that $\hFnr$ is a subcategory of $\hFn$.
Let $\hFnrchi$ be the subcategory of $\hFnr$ consisting of finite dimensional $\afSrv$-modules with elliptic character $\chi$.

\subsection{}
Proposition \ref{indecomposable module} and Proposition \ref{Q=Q'} are the key to the proof of Theorem \ref{block}. To prove Proposition \ref{indecomposable module} we need the following two lemmas.
\begin{Lem}\label{Weyl modules belonging to hFnrQ}
(1) For $\bfQ\in\Qnpr$, $W(\bfQ)\in\hFnrQ$.

(2) We have $\hFnchio\ot\hFnchit\han\hFnchiot$.
\end{Lem}
\begin{proof}
According to Corollary \ref{Weyl module for affine q-Schur algebras}, we have
$W(\bfQ)\in\hFnr$ for $\bfQ\in\Qnpr$. If $L(\bfQ')$ is a composition factor of $W(\bfQ)$, then by Proposition \ref{l-weight for l highest weight module}, we have $\bfQ'\in\bfQ\Rnm$. This implies that $\ol{\bfQ}=\ol{\bfQ'}$. The assertion (1) follows. To prove (2), it is enough to prove
that $L(\bfQ_1)\ot L(\bfQ_2)\in\hFnQot$ for $\bfQ_1,\bfQ_2\in\Qnp$. Suppose that $L(\bfQ)$ is a composition factor of $L(\bfQ_1)\ot L(\bfQ_2)$. Then by \cite[Lem. 4.1]{FM} and Proposition \ref{l-weight for l highest weight module} we conclude that $\wtl(L(\bfQ))\han\wtl(L(\bfQ_1))\wtl(L(\bfQ_2))\han\bfQ_1\bfQ_2\Rnm$. Thus $\ol{\bfQ}=\ol{\bfQ_1}\cdot\ol{\bfQ_2}$. The assertion (2) follows.
\end{proof}

We define an algebra anti-automorphism $\tau:\afUglv\ra\afUglv$ by
$$\tau(\ttx_{j,s}^\pm)=\ttx_{j,s}^\mp,\quad\tau(\ttk_i^{\pm 1})=
\ttk_i^{\pm 1},\quad\tau(\ttg_{i,t})=\ttg_{i,t}$$
for $1\leq j\leq n-1$, $s\in\mbz$, $1\leq i\leq n$ and $t\in\mbz\backslash\{0\}$.
Given a finite dimensional left $\afUglv$-module $V$, we write $V^\circ$ for the dual space $\Hom_{\mbc}(V,\mbc)$ regarded as a left $\afUglv$-module via the action
$(hf)(x)=f(\tau(h)x)$, for $h\in\afUglv$, $f\in\Hom_{\mbc}(V,\mbc)$ and $x\in V$. Since
$\tau(\ttk_i^{\pm 1})=\ttk_i^{\pm 1}$ and $\tau(\ttg_{i,t})=\ttg_{i,t}$, we have $\ch(V)=\ch(V^\circ)$. This implies that
\begin{equation}\label{dual of irreducible module}
 L(\bfQ)^\circ\cong L(\bfQ)
\end{equation}
 for $\bfQ\in\Qnp$.
For $M_1,M_2\in\hFn$ we have
$$\Ext^1_{\hFn}(M_1,M_2)\cong\Ext^1_{\hFn}(M_2^\circ,M_1^\circ).$$
Thus
$$\Ext^1_{\hFn}(L(\bfQ),L(\bfQ'))\cong\Ext^1_{\hFn}(L(\bfQ'),L(\bfQ))$$
for $\bfQ,\bfQ'\in\Qnp$.

\begin{Lem}\label{extension between different chi}
(1) Let $W\in\hFnchi$ and $\bfQ_0\in\Qnp$ with $\chi\not=\ol{\bfQ_0}$. Then
$\Ext^1_{\hFn}(W,L(\bfQ_0))=0$.

(2) Assume that $V_i\in\hFnchii$, $i=1,2$ and that $\chi_1\not=\chi_2$. Then
$\Ext^1_{\hFn}(V_1,V_2)=0$.
\end{Lem}
\begin{proof}
Let $\ell(W)$ be the length of $W$. To prove (1), we proceed by induction on $\ell(W)$.  Suppose first that $\ell(W)=1$. Then $W=L(\bfQ)$ for some $\bfQ\in\Qnp$.
Consider a short exact sequence
\begin{equation}\label{exact sequence}
 0\lra L(\bfQ_0)\stackrel{f}{\lra} V\stackrel{g}{\lra} L(\bfQ)\lra 0.
\end{equation}
We have to show that the short exact sequence is split. Let $\la=\deg\bfQ$ and $\mu=\deg\bfQ_0$.
Then one of the  following hold,
\begin{itemize}
\item[(i)] $\la \lhd\mu$, or
\item[(ii)] $\la\not\!\!\lhd\,\mu$.
\end{itemize}
From \eqref{dual of irreducible module} we see that the short exact sequence \eqref{exact sequence} yields a short exact sequence
$$0\lra L(\bfQ){\lra} V^\circ{\lra} L(\bfQ_0)\lra 0.$$
Thus we may assume that $\la\not\!\!\lhd\,\mu$ without loss of generality.
For $1\leq i\leq n-1$ let $\al_i=\bse_i-\bse_{i+1}$, where $\bse_i$ is defined in \eqref{ei}.
Since $\la\not\!\!\lhd\,\mu$ and $L(\bfQ_0)=\oplus_{\nu\unlhd\mu}L(\bfQ_0)_\nu$ we have
$L(\bfQ_0)_{\la+\al_i}=0$ for $1\leq i\leq n-1$. This implies that $\dim V_{\la+\al_i}=
\dim L(\bfQ)_{\la+\al_i}+\dim L(\bfQ_0)_{\la+\al_i}=0$ and hence
\begin{equation}\label{action of xis+ on Vla}
\ttx_{i,s}^+V_\la=0
\end{equation}
 for $1\leq i\leq n-1$ and $s\in\mbz$.
Since the elements $\ms Q_{i,s}$ ($1\leq i\leq n$, $s\in\mbz$) commute among themselves, there exists an element $0\not= w\in V_{\bfQ}\han V_\la$ such that
$$\ms Q_{i,s}w=Q_{i,s}w$$
for $1\leq i\leq n$ and $s\in\mbz$. This, together with Lemma \ref{l-highest weight module} and \eqref{action of xis+ on Vla}, implies that
$\afUglv w$ is a quotient of $W(\bfQ)$. It follows from Lemma \ref{Weyl modules belonging to hFnrQ} that $\afUglv w\in\hFnQ$. Thus, since $\ol\bfQ\not=\ol{\bfQ_0}$, we have $f(L(\bfQ_0))\not\han\afUglv w$. This implies that $f(L(\bfQ_0))\cap\afUglv w=0$. Therefore,
we have $$V=f(L(\bfQ_0))\oplus\afUglv w.$$
This shows that the short exact sequence \eqref{exact sequence} is split.

Now we assume that $\ell(W)>1$. Let $W_1$ be a proper nontrivial submodule of $W$ and let $W_2=W/W_1$. Then by induction we have $\Ext^1_{\hFn}(W_1,L(\bfQ_0))=\Ext^1_{\hFn}(W_2,L(\bfQ_0))=0$. Furthermore the short exact sequence
$$ 0\lra W_1\lra W\lra W_2\lra 0$$
yields a short exact sequence
$$\Ext^1_{\hFn}(W_2,L(\bfQ_0))\lra\Ext^1_{\hFn}(W,L(\bfQ_0))
\lra\Ext^1_{\hFn}(W_1,L(\bfQ_0)).$$
 Thus
$\Ext^1_{\hFn}(W,L(\bfQ_0))=0$. The assertion (1) follows.
The assertion (2) is proved similarly by induction on $\ell(V_2)$.
\end{proof}

\begin{Prop}\label{indecomposable module}
Let $V$ be an indecomposable $\afUglv$-module in $\hFn$. Then $V\in\hFnQ$ for some $\bfQ\in\Qnp$.
\end{Prop}
\begin{proof}
The proof is similar to that of \cite[Th. 8.3(i)]{CM}.
We proceed by induction on $\ell(V)$. If $\ell(V)=1$ then $V=L(\bfQ)$ for some $\bfQ\in\Qnp$. By definition we have $V\in\hFnQ$.
Now we assume $\ell(V)>1$. Let $L(\bfQ_0)$ be an irreducible submodule of $V$ and let $W=V/L(\bfQ_0)$. Write $W=\oplus_{1\leq j\leq s}W_j$, where $W_j$ is indecomposable for all $j$. By induction, for each $1\leq j\leq r$, there exist $\chi_j\in\Xn/\Rn$ such that $W_j\in\hFnchij$. We have to show that $\bfQ_0=\chi_j$ for all $1\leq j\leq s$. If this is not true, then there exists $j_0$ such that $\chi_{j_0}\not=\ol{\bfQ_0}$.
By Lemma \ref{extension between different chi} we have
$\Ext^1_{\hFn}(W,L(\bfQ_0))\cong\Ext^1_{\hFn}(\oplus_{j\not=j_0}W_j,L(\bfQ_0))$. It follows that there exist a short exact sequence
$ 0\ra L(\bfQ_0)\ra V'\ra \oplus_{j\not=j_0}W_j \ra 0$ in $\Ext^1_{\hFn}(\oplus_{j\not=j_0}W_j,L(\bfQ_0))$
such that the short exact sequence
$ 0\ra L(\bfQ_0)\ra V\ra W\ra 0$
is equivalent to the short exact sequence $ 0\ra L(\bfQ_0)\ra V'\oplus W_{j_0}\ra \oplus_{j\not=j_0}W_j \oplus W_{j_0}\ra 0$. In particular, we have
$V\cong V'\oplus W_{j_0}$, which is a contradiction. Thus
$\bfQ_0=\chi_j$ for all $j$ and $V\in\hFnQz$.
\end{proof}

\subsection{}
Before proving Proposition \ref{Q=Q'}, we need several preliminary lemmas.
According to \cite{AK,Chari,VV02} we have the following result.
\begin{Lem}\label{cyclic}
Let $1\leq i_1,\cdots,i_r\leq n$, $a_1,\cdots,a_r\in\mbc^*$. Then there exist a permutation
$\sg\in\fSr$ such that $L(\bfQ_{i_1,a_{\sg(1)}})\ot\cdots \ot L(\bfQ_{i_r,a_{\sg(r)}})$ is cyclic on the tensor product of the $\ell$-highest weight vectors.
\end{Lem}

\begin{Lem}\label{commutative}
Let $V_i\in\hFn$ ($1\leq i\leq r$). Then $V_1\ot\cdots\ot V_r$ and $V_{\sg(1)}\ot\cdots \ot  V_{\sg(r)}$ have the same composition factors for all $\sg\in\fSr$.
\end{Lem}
\begin{proof}
Let $\Rep\,\afUglv$ be the Grothendieck ring of the category $\hFn$. For $X\in\hFn$ we write $[X]$ for the class of $X$ in $\Rep\,\afUglv$. According
to \cite[Lem. 4.5]{FM} we know that $\Rep\,\afUglv$ is a commutative ring. It follows that for any $\sg\in\fSr$ we have $[V_1\ot\cdots\ot V_r]=[V_1]\cdots[V_r]=[V_{\sg(1)}\ot\cdots V_{\sg(r)}]$. The assertion follows.
\end{proof}

For $V_1,V_2\in\hFn$, we write
$V_1\sim_n V_2$
if $V_1$ and $V_2$ belong to the same block of the category $\hFn$. Similarly, for $V_1,V_2\in\hFnr$, we write
$V_1\sim_{n,r} V_2$
if $V_1$ and $V_2$ belong to the same block of the category  $\hFnr$.
For $a\in\mbc^*$ let
$L_a=L(\bfQ_{1,a}).$

\begin{Lem}\label{link}
Let $a_1,\cdots,a_r\in\mbc^*$ and let $\sg$ be an element in the Symmetric
group $\fSr$. Then $L_{a_1}\ot\cdots\ot L_{a_r}\sim_n L_{a_{\sg(1)}}\ot\cdots\ot L_{a_{\sg(r)}}$ and $L_{a_1}\ot\cdots\ot L_{a_r}\sim_{n,r} L_{a_{\sg(1)}}\ot\cdots\ot L_{a_{\sg(r)}}$.
\end{Lem}
\begin{proof}
According to Lemma \ref{tensor of simple module}, for any $\sg\in\fSr$, $L_{a_{\sg(1)}}\ot\cdots\ot L_{a_{\sg(r)}}$ can be regarded as an $\afSrv$-module via $\zeta_r$. Thus by Lemma \ref{cyclic}, we conclude that there exist a permutation $\pi\in\fSr$ such that $L_{a_{\pi(1)}}\ot\cdots L_{a_{\pi(r)}}$ is an indecomposable $\afSrv$-module. Now the assertion follows from Lemma \ref{commutative}.
\end{proof}

\begin{Prop}\label{Q=Q'}
Let $\bfQ\in\Qnpr$ and $\bfQ'\in\Qnpt$. Assume that $\ol{\bfQ}=\ol{\bfQ'}$. Then
$r=t$, $L(\bfQ)\sim_n L(\bfQ')$ and $L(\bfQ)\sim_{n,r} L(\bfQ')$.
\end{Prop}
\begin{proof}
By Lemma \ref{Decomposition of L(bfQ)}, Lemma \ref{cyclic} and Lemma \ref{commutative}  we conclude that there exist $a_1,\cdots,a_r\in\mbc^*$ such that
$L(\bfQ)$ is the composition factor of $L_{a_1}\ot\cdots \ot L_{a_r}$ and $L_{a_1}\ot\cdots \ot L_{a_r}$ is indecomposable. Similarly, we choose $b_1,\cdots,b_t\in\mbc^*$ such that
$L(\bfQ')$ is the composition factor of $L_{b_1}\ot\cdots \ot L_{b_t}$ and $L_{b_1}\ot\cdots \ot L_{b_t}$ is indecomposable. From Proposition
\ref{indecomposable module} we see that $$\ol{\La_{1,a_1}}\cdots\ol{\La_{1,{a_r}}}=\ol{\bfQ}=\ol{\bfQ'}
=\ol{\La_{1,b_1}}\cdots\ol{\La_{1,{b_t}}}.$$
It follows from Proposition \ref{vartheta iso}(1) that $r=t$ and there exists $\sg\in\fSr$ such that $(b_1,\cdots,b_r)=(a_{\sg(1)},\cdots,a_{\sg (r)})$. Now the assertion follows from Lemma \ref{link}.
\end{proof}

\subsection{}
Recall from \eqref{Xinr} the definition of $\Xi_n$ and $\Xi_{n,r}$.
We are now prepared to prove the main result of this paper.
\begin{Thm}\label{block}
We have
$$\hFn=\bop_{\chi\in\Xi_n}\hFnchi,\quad\hFnr=\bop_{\chi\in\Xi_{n,r}}\hFnrchi.$$
Furthermore each $\hFnchi$ is a block of $\hFn$ for $\chi\in\Xi_n$ and each
$\hFnrchi$ is a block of $\hFnr$ for $\chi\in\Xi_{n,r}$.
\end{Thm}
\begin{proof}
From Proposition \ref{indecomposable module} and \cite[Lem. 2.5]{CM04} we see that
each block of $\hFn$ is contained in $\hFnchi$ for some $\chi\in\Xi_n$ and
each block of $\hFnr$ is contained in $\hFnrchi$ for some $\chi\in\Xi_{n,r}$. Now the assertion follows from Proposition \ref{Q=Q'}.
\end{proof}

Combining Proposition  \ref{vartheta iso} with Theorem \ref{block} yields the following result.
\begin{Thm}\label{equivalent condition}
Let $\bfQ=(Q_1(u),\cdots,Q_n(u))\in\Qnpr$ and $\bfQ'=(Q_1'(u),\cdots,Q_n'(u))\in\Qnpt$. Then the following are equivalent:

(i) $L(\bfQ),L(\bfQ')$ belong to the same block of $\hFn$;

(ii) $r=t$, and  $L(\bfQ),L(\bfQ')$ belong to the same block of $\hFnr$;

(iii) $\ol{\bfQ}=\ol{\bfQ'}$;

(iv) $r=t$ and $\prod_{1\leq i\leq n}Q_i(u)=\prod_{1\leq i\leq n}Q_i'(u)$.
\end{Thm}

\section{Blocks of affine Hecke algebras}

\subsection{}
The blocks of extended affine Hecke algebras of type $A$ are known (see \cite[Th. 7]{Muller}, \cite[III.9]{BG} and \cite[Th. 2.15]{LM}).
We will use Theorem \ref{equivalent condition} to give a different approach to the classification of the blocks for the affine Hecke algebra $\afHrv$.

A {\it segment} $\sfs$ with center $a\in\mbc^*$ is by definition an
ordered sequence
$$\sfs=(a\ttv^{-k+1},a\ttv^{-k+3},\ldots,a\ttv^{k-1})\in(\mbc^*)^k.$$
 Here $k$ is called the length
of the segment, denoted by $|\sfs|$. If
$\bfs=\{\sfs_1,\ldots,\sfs_p\}$ is an unordered collection of
segments, let $|\bfs|=|\sfs_1|+\ldots+|\sfs_p|$. We also call $|\bfs|$ the
length of $\bfs$.
Let $\mathscr S_r$ be the set of unordered collections of segments
$\bfs$ with $|\bfs|=r$.

If $w=s_{i_1}s_{i_2}\cdots s_{i_m}\in\fSr$ is reduced let $T_w=T_{i_1}T_{i_2}\cdots T_{i_m}$. Let $\Hrv$ be the subalgebra of $\afHrv$ generated by $T_i$ for $1\leq i\leq r-1$.
For $\mu\in\La(p,r)$ let $\frak S_\mu$ be the corresponding standard Young subgroup of  $\frak S_r$ and let $y_{\mu}=\sum_{w\in\fS_\mu}(-\ttv^2)^{-\ell(w)}T_w\in\Hrv$.
Then,
\begin{equation*}\label{signed permutation module}
\Hrv y_\mu\cong E_\mu\oplus(\bigoplus_{\nu\vdash r,
\nu\rhd\la}m_{\nu,\mu}E_\nu),
\end{equation*}
where $E_\nu$ is the left cell module defined by the
Kazhdan--Lusztig's C-basis \cite{KL79} associated with the left cell
containing $w_{0,\nu}$ and $w_{0,\nu}$ is the longest element in $\fS_\nu$.

For
$\bfb=(b_1,\ldots,b_r)\in(\mbc^*)^r$, let $$M_{\bfb}=\afHrv/J_\bfb,$$
where $J_\bfb$ is the left ideal of $\afHrv$ generated by $X_j-b_j$
for $1\leq j\leq r$.
Given $\mu=(\mu_1,\cdots,\mu_p)\in\La(p,r)$ and $\bsa=(a_1,\cdots,a_p)\in(\mbc^*)^p$ with $\mu_i\geq 1$ for all $i$, let $\bfb=(\sfs_1,\cdots,\sfs_p)\in(\mbc^*)^r$, where $\sfs_i=(a_i\ttv^{-\mu_i+1},a_i\ttv^{-\mu_i+3},
\ldots,a_i\ttv^{\mu_i-1})$. Let
$$I_{\mu,\bsa}=\afHrv\bar y_\mu\han M_\bfb $$
where $\bar y_\mu=y_\mu+J_\bfb\in M_{\bfb}$.
Note that $I_{\mu,\bsa}$ is isomorphic to $\Hrv y_\mu$ as an $\Hrv$-module.

Given $\bfs=\{\sfs_1,\ldots,\sfs_p\}\in\ms S_r$ with $\sfs_i=(a_i\ttv^{-\mu_i+1},a_i\ttv^{-\mu_i+3},\ldots,a_i\ttv^{\mu_i-1})\in(\mbc^*)^{\mu_i},$ let $\mu=(\mu_1,\cdots,\mu_p)\in\La(p,r)$ and let $\bsa=(a_1,\cdots,a_p)\in(\mbc^*)^p$.
Let $V_\bfs$ be the unique composition factor of the $\afHrv$-module
$I_{\mu,\bsa}$ such that the multiplicity of $E_{\mu}$ in
$V_\bfs$ as an $\Hrv$-module is nonzero.
The following classification theorem is due to
Zelevinsky \cite{Zelevinsky} and Rogawski \cite{Rogawski}.

\begin{Thm}\label{classification irr affine Hecke algebra}
The modules $ V_\bfs$ with $\bfs\in\mathscr S_r$ are all nonisomorphic finite dimensional
irreducible $\afHrv$-modules.
\end{Thm}

\subsection{}
The $\afHrv$-modules $I_{\mu,\bsa}$ and the Weyl modules $W(\bfQ)$ for the affine quantum Schur algebra $\afSrv$ can be related by a functor $\sfF$, which we now describe.
Recall from \S6 that $\hFnr$ is the category of finite dimensional $\afSrv$-modules. Let $\Cr$  be the category of finite dimensional $\afHrv$-modules. Using the $\afSrv$-$\afHrv$-bimodule $\Ogv^{\ot r}$, we define a functor \begin{equation*}\label{functor sF}
\sfF:\Cr\lra\hFnr,\;V\longmapsto\Ogv^{\ot
r}\ot_{\afHrv}V.
\end{equation*}
If $n\geq r$ then the functor $\sfF$ is an equivalence of categories
(see \cite[Th. 4.2]{CP96} and \cite[Th. 4.1.3]{DDF}).
According to \cite[Prop. 4.6]{Fu12} we have the following results.
\begin{Lem}\label{prop of Schur functor}
Let $\mu=(\mu_1,\cdots,\mu_p)\in\La(p,r)$ and $\bsa=(a_1,\cdots,a_p)\in(\mbc^*)^p$ with $\mu_i\geq 1$ for all $i$. Then
$\sfF(I_{\mu,\bsa})$ is isomorphic to $L(\bfQ_{\mu_1,a_1})\ot\cdots\ot L(\bfQ_{\mu_p,a_p})$ if $\mu_i\leq n$ for all $i$, and it is zero, otherwise.
\end{Lem}

Combining Proposition \ref{Decomposition of Weyl modules for affine gln} with Lemma \ref{prop of Schur functor} yields the following result.

\begin{Coro}\label{Verma module and Weyl module}
For $\bfQ\in\Qnpr$ there exist $1\leq\mu_1,\cdots,\mu_p\leq n$ and $a_1,\cdots,a_p\in\mbc^*$ such that $W(\bfQ)\cong \sfF(I_{\mu,\bsa})$, where
$\mu=(\mu_1,\cdots,\mu_p)$ and $\bsa=(a_1,\cdots,a_p)$.
\end{Coro}

\subsection{}
Assume $n>r$. For $\bfs=\{\sfs_1,\ldots,\sfs_p\}\in\ms S_r$ with
$$\sfs_i=(a_i\ttv^{-\mu_i+1},a_i\ttv^{-\mu_i+3},\ldots,a_i\ttv^{\mu_i-1})\in(\mbc^*)^{\mu_i},$$
define
\begin{equation}\label{partiall}
\bfQ_\bfs=(Q_1(u),\ldots,Q_n(u))
\end{equation}
by setting recursively
\begin{equation*}
Q_i(u)=\begin{cases} 1,\quad&\text{ if }i=n;\\
P_i(u\ttv^{-i+1})P_{i+1}(u\ttv^{-i+2})\cdots
P_{n-1}(u\ttv^{n-2i}),&\text{ if }n-1\geq i\geq1,\\\end{cases}\\
\end{equation*}
where  $P_{i}(u)=\prod_{1\leq j\leq p\atop \mu_j=i}(1-a_ju)$.

Given $\bfs=\{\sfs_1,\ldots,\sfs_p\}\in\mathscr S_{r}$, let
\begin{equation}\label{bfs(bfs)}
\bfa(\bfs)=(\sfs_1,\ldots,\sfs_p)\in(\mbc^*)^r
\end{equation}
be the $r$-tuple obtained by juxtaposing the segments in $\bfs$.
The symmetric group $\fSr$ acts
on the set $(\mbc^*)^r$ by place permutation:
\begin{equation*}\label{place permutation}
\bfb w=(b_{w(1)},\cdots,b_{w(r)}),\quad\text{
for $\bfb=(b_1,\cdots,b_r)\in (\mbc^*)^r$ and $w\in\fSr$.}
\end{equation*}

\begin{Thm}\label{block for affine Hecke algebras}
Let $\bfs,\bfs'\in\mathscr S_r$. Then $V_\bfs$ and $V_{\bfs'}$ belong to the same block of $\Cr$ if and only if there exists some $\sg\in\fSr$ such that
$\bfa(\bfs)\sg=\bfa(\bfs')$, where $\bfa(\bfs)$ and $\bfa(\bfs')$ are defined using \eqref{bfs(bfs)}.
\end{Thm}
\begin{proof}
Assume $\bfs=\{\sfs_1,\sfs_2,\cdots,\sfs_k\}$ and
$\bfs'=\{\sfs_1',\sfs_2',\cdots,\sfs_l'\}$ with
$$\sfs_i=(a_i\ttv^{-\mu_i+1},a_i\ttv^{-\mu_i+3},
\ldots,a_i\ttv^{\mu_i-1}),\quad
\sfs_j'=(a_j'\ttv^{-\mu_j'+1},a_j'\ttv^{-\mu_j'+3},
\ldots,a_j'\ttv^{\mu_j'-1}).$$
Choose $n$ such that $n>r$. Then by \cite[4.4.2]{DDF}  we conclude that $V_\bfs$ and $V_{\bfs'}$ belong to the same block of $\Cr$ if and only if $L(\bfQ_\bfs)$ and $L(\bfQ_{\bfs'})$ belong to the same block of $\hFnr$, where $\bfQ_\bfs=(Q_1(u),\cdots,Q_n(u))$ and $\bfQ_{\bfs'}=(Q_1'(u),\cdots,Q_n'(u))$ are defined using \eqref{partiall}. Furthermore, if we write  $\bfa(\bfs)=(b_1,\cdots,b_r)$ and $\bfa(\bfs')=(b_1',\cdots,b_r')$,
then
\begin{equation*}
\begin{split}
\prod_{1\leq i\leq n}Q_i(u)&=\prod_{1\leq i\leq k\atop 1\leq s\leq\mu_i}(1-a_iuv^{2s-1-\mu_i})=\prod_{1\leq s\leq r}(1-b_su),\\
\quad \prod_{1\leq i\leq n}Q_i'(u)&=\prod_{1\leq i\leq l\atop 1\leq s\leq\mu_i'}(1-a_i'uv^{2s-1-\mu_i'})=\prod_{1\leq s\leq r}(1-b_s'u).
\end{split}
\end{equation*}
Thus by Theorem \ref{equivalent condition} we conclude that
$V_\bfs$ and $V_{\bfs'}$ belong to the same block of $\Cr$ if and only if there exists some $\sg\in\fSr$ such that $(b_1',\cdots,b_r')=(b_{\sg(1)},\cdots,b_{\sg(r)})$. The proof is completed.
\end{proof}


\end{document}